\documentclass[11pt]{amsart}
\usepackage{amsmath,amsfonts,amssymb,amsthm,amscd,latexsym,euscript,verbatim}
\usepackage[all]{xy}
\usepackage{graphicx}
\usepackage{color}
\makeatletter
\def\section{\@startsection{section}{1}%
  \z@{1.1\linespacing\@plus\linespacing}{.8\linespacing}%
  {\normalfont\Large\scshape\centering}}
\makeatother

\theoremstyle{plain}

\newtheorem*{conj*}{Root Groups Conjecture}
\newtheorem*{thm1.2}{(1.2) Theorem}
\newtheorem*{thm1.3}{(1.3) Theorem}
\newtheorem*{thm1.4}{(1.4) Theorem}
\newtheorem*{prop*}{Proposition}

\newtheorem{prop}{Proposition}[section]

\newtheorem{thm}[prop]{Theorem}
\newtheorem{cor}[prop]{Corollary}
\newtheorem{lemma}[prop]{Lemma}

\theoremstyle{definition}

\newtheorem{definition}[prop]{Definition}
\newtheorem{hypothesis}[prop]{Hypothesis}
\newtheorem*{Def*}{Definition}

\newtheorem{notation}[prop]{Notation}
\newtheorem*{notation*}{Notation}
\newtheorem{remark}[prop]{Remark}

\newcommand{\Cen}{\rm Cen}
\newcommand{\cala}{\mathcal{A}}

\newcommand{\calg}{\mathcal{G}}

\newcommand{\qq}{\mathbb{Q}}

\newcommand{\zz}{\mathbb{Z}}
\newcommand{\fix}{\mathrm{fix}}

\newcommand{\ga}{\alpha}
\newcommand{\gb}{\beta}
\newcommand{\gc}{\gamma}

\renewcommand{\gg}{\gamma}

\newcommand{\gd}{\delta}

\newcommand{\gre}{\epsilon}
\newcommand{\gl}{\lambda}

\newcommand{\sminus}{\setminus}
\newcommand{\lan}{\langle}
\newcommand{\ran}{\rangle}

\newcommand{\Tr}{{\rm Tr}}

\numberwithin{equation}{section}

\begin{document}
\title[A non-split sharply $2$-transitive group]{Sharply $2$-transitive groups
in characteristic~$0$}
\author[Eliyahu Rips,  Katrin Tent]{Eliyahu Rips$^1$\qquad Katrin Tent$^2$}

\address{Eliyahu Rips\\
       Einstein Institute of Mathematics\\
        Hebrew University \\
        Jerusalem 91904\\
        Israel}
\email{eliyahu.rips@mail.huji.ac.il}
\thanks{$^1$This research was partially supported by the Israel Science Foundation}

\address{Katrin Tent \\
         Mathematisches Institut \\
         Universit\"at M\"unster \\
	 Einsteinstrasse 62\\
         48149 M\"unster \\
         Germany}
\email{tent@wwu.de}
\thanks{$^2$Partially supported by SFB 878}


\begin{abstract}
We construct sharply $2$-transitive groups of characteristic~$0$ without non-trivial abelian
normal subgroups. These groups act sharply $2$-transitively by conjugation on their involutions. This answers a long-standing open question. 
\end{abstract}

\date{\today}
\maketitle
\section{Introduction}

The finite sharply $2$-transitive groups were classified by Zassenhaus~\cite{Z} in the 1930's. They were shown to all
contain abelian normal subgroups and thus arise from nearfields
in a way similar to $AGL(1,K)$ from a field~$K$.
It remained an open question whether arbitrary sharply $2$-transitive groups necessarily contain an abelian normal subgroup and the first counterexamples were given by \cite{RST}. In the examples constructed there,  involutions have no
fixed points and
are said to have 'characteristic $2$', leaving open the existence of non-split sharply $2$-transitive groups in other characteristics. The situation when involutions of a sharply $2$-transitive group do have fixed points is rather different from the fixed point free setting. In contrast to the result in \cite{RST} which shows that \emph{any} group can be extended to a sharply $2$-transitive group in which involutions are fixed point free, it here turns out that a group has to satisfy a number of necessary conditions in order to be embeddable into a sharply $2$-transitive group in which involutions have fixed points.

Suppose that $G$ acts sharply $2$-transitively on a set $X$. It is easy to see that the set of involutions $J$ of a sharply $2$-transitive group $G$ forms a conjugacy class. If involutions have fixed points, then it is also not hard to see that there is a $G$-equivariant bijection between $J$ and $X$ taking each involution $\iota$ to its unique fixed point $\fix(\iota)$.

Thus, if involutions have fixed points, then $G$ acts sharply $2$-transitively on $J$, implying that all dihedral subgroups
of $G$ are isomorphic, see e.g. \cite{RST} and \cite{Te2} for more background.

 We here construct sharply $2$-transitive groups of characteristic $0$. Here characteristic $0$ means that involutions have fixed points and the product of distinct involutions has infinite order, as is the case in $AGL(1,K)$ for a field of characteristic $0$.

In order to formulate our main result, we introduce the following terminology:
given an involution $j\in G$ we say that involutions
$s,t\in G$ are \emph{equivalent relative to $j$} (and write $s\sim_jt$) if
there are $n,m\in\mathbb{Z}$ such that $(js)^m=(jt)^n$. Note that any involution $s\in\langle j,t\rangle\setminus \{j\}$ is equivalent to $t$ relative to $j$: 
we have $s=(jt)^nj$ or $s=j(jt)^n$ for some $n\in \mathbb{N}$. Hence $js=(jt)^{-n}$ or
$js=(jt)^n$, respectively.
Since $j$ will be fixed throughout this paper, we mostly omit mentioning $j$ in the equivalence relation.
We call an involution $s_1$ \emph{minimal}
if for any involution $s_2$ equivalent to $s_1$ (relative to $j$)
we have $\lan j,s_1\ran\geq \lan j,s_2\ran$.

For any involution $s\in G$, we write $\overline{\lan j,s\ran}^G=\bigcup_{s'\sim_j s}\lan j,s'\ran$. Hence, if $s\in G$ is minimal relative to $j$,
then $\overline{\lan j, s\ran}^G=\lan j,s\ran$. We omit mentioning $G$ if the ambient group is clear from the context.
By a \emph{translation} we will mean a product of two involutions. For a set $A\subset G$ we write $\Tr(A)$ for the set of translations in $A$. Note that we have $\Tr(\overline{\lan j,s\ran})=\{1\}\cup \{js'\colon s' \mbox{ equivalent to } s\}$.

\medskip
Our main result is the following:


\begin{thm}\label{thm s2t in char 0}

Let $G$ be a group containing involutions $j,t$ and $t'$ with $j^{t'}=t$.
Let $A={\Cen}_G(j)$ and assume  that the following holds:
\begin{enumerate} 
\item all involutions in $G$ are conjugate;

\item any two distinct involutions of $G$ generate infinite dihedral groups;

\item for any $n>0$ there is a unique involution $s\in G$ with
$(js)^n=jt$.

\item for any involution $s$ equivalent to $t$  there is some $a\in A$ such that $s=t^a$.

\item for any involution $s$, we have ${\Cen}(js)=\Tr(\overline{\lan j,s\ran}^G)$.

\item for any involution $s\notin AtA$ there is a
unique minimal involution $s'$ equivalent to $s$ relative to $j$.

\end{enumerate}
Then $G$ is contained in a group $\calg$ 
acting sharply $2$-transitively and in characteristic~$0$ on the set of involutions of $\calg$ and such that $\calg$ contains no abelian normal subgroup.

\end{thm}

As an immediate consequence of Theorem \ref{thm s2t in char 0} we have

\begin{thm}\label{t:existence}
There exist sharply $2$-transitive groups of characteristic~$0$ not containing a non-trivial abelian normal subgroup.
\end{thm}

Theorem~\ref{thm s2t in char 0} will follow inductively from
the following proposition, which states that any Zassenhaus group
can be extended to a Zassenhaus group in which the point stabilizers
are `a bit more transitive':

\begin{prop}\label{p:main}

Let $G$ be a group containing involutions $j,t$ and $t'$ with $j^{t'}=t$.
Let $A={\Cen}_G(j)$ and assume that conditions $(1)$ -- $(6)$ of
Theorem~\ref{thm s2t in char 0} hold.

Then for any involution $v\in G$ with $v\neq j$ there exist  an extension $G_1$ of $G$ such that for $A_1={\Cen}_{G_1}(j)$  there exists some $f\in A_1$ with $t^f=v$ and conditions $(1) - (6)$ continue to hold with $G_1$ and $A_1$ in place of $G,A$.
\end{prop}

The main body of the paper is concerned with the proof of Proposition~\ref{p:main}. Theorems~\ref{thm s2t in char 0} and \ref{t:existence} then follow rather quickly from this. Their proofs
will be given in Section~\ref{s:maintheorems}.

It might be worth pointing out that for a group $G$ to be contained in a sharply $2$-transitive group of characteristic $0$, conditions (2) and (5) of Theorem~\ref{thm s2t in char 0} are necessary, while assumption (3) is
necessary in the following weaker form:

\medskip

(3') {\it for any $n>0$ there is at most one involution $s\in G$ with $(js)^n=jt$.}
\medskip

Assumptions (1), (4) and (6) will necessarily hold in any sharply $2$-transitive group of characteristic $0$.  They are added here as prerequisites so as to facilitate the proof of Proposition~\ref{p:main}, which forms the induction step for the proof of Theorem~\ref{thm s2t in char 0}.

\medskip
 We would like to point out that in the case where involutions have fixed points, the approach using partial actions as in \cite{TZ, Te1} leads to exactly the same situation
to deal with as in the approach used here: since any sharply $2$-transitive action of $G$
on a set $X$ in characteristic different from $2$ is necessarily (equivalent to) the conjugation action on the set of involutions of $G$ and the involutions generate subgroups, one does not gain any freedom in the construction from using partial actions.

Finally we mention that while the sharply $2$-transitive groups of characterstic $2$ constructed in \cite{TZ} where shown in \cite{Te1} to be point stabilizers of sharply $3$-transitive groups (of the same characteristic) we do not know whether such an extension is possible for the groups constructed here.

\section{Some preliminaries on involutions}\label{sect explanation}

In this section we  collect some facts about involutions and
the equivalence relation relative to the fixed involution $j\in A$ defined above.
Throughout this section we assume that $G,j, A={\Cen}_G(j), t$ and $t'$ are as in 
the statement of Theorem~\ref{thm s2t in char 0} and satisfy the assumptions (1) -- (6) given there.

Recall that a proper subgroup $B$ of a group
$H$ is {\it malnormal} in $H$ if
 $B\cap g^{-1}Bg=1,$ for all $g\in H\setminus B$.

\begin{lemma}\label{l:malnormal}\label{l:invol}
(i) The centralizer $A={\Cen}(j)$ is malnormal in $G$.

(ii) If $g^n\in A$, then $g\in A$ or $g^n=1$.

(iii)  The intersection of $A$ with any dihedral subgroup
of $G$ is contained in $\lan j\ran$. In particular, $j$ is the unique involution
contained in $A$.

(iv) If $r, s\in G\sminus A$ are involutions with $s=arb$ for some $a,b\in A$, then $a=b^{-1}$.
\end{lemma}
\begin{proof}
(i)  Suppose $g\in G\setminus A$ is such that $A\cap A^g\neq 1$.
Since $A^g={\Cen}_{G}(j^g)$, any $h\in A\cap A^g$
centralizes both $j$ and $j^g\neq j$. Hence $h$ centralizes
the translation $jj^g$ and therefore is contained in $\Tr\overline{\lan j,j^g\ran}$. By Property (5) no nontrivial translation in  $\Tr\overline{\lan j,j^g\ran}$ centralizes $j$, a contradiction.

(ii) If $g^n\in A$, then $g^n\in A\cap A^g$, and thus by malnormality of $A$
we have $g^n=1$ or $g\in A$.

(iii) Since $A={\Cen}_G(j)$ is malnormal and all involutions are conjugate in $G$, it follows that $j$ is the unique involution in $A$.
Now if $s, r$ are distinct involutions in $G$
such that $sr\in A$, we may assume  that $s\notin A$.
Since $s$ inverts $rs,$ we contradict the malnormality of $A$
in $G$.  

(iv) We have $s=ara^{-1}ab$.  Hence $ab=ara^{-1}s$ is a product of two involutions.
By (iii), $ab=1$.
%
\end{proof}

\begin{lemma}\label{l:translation_product}
If  $1\neq  s_1s_2\in\lan j,s_3\ran$ for involutions $s_1,s_2,s_3\neq j$,
then $s_1 $ and $s_2$ are equivalent to $s_3$ relative to~$j$.
\end{lemma}
\begin{proof}
We may suppose $s_1s_2=js_3$. 
Conjugation by $j$ inverts both sides and hence yields $s_1^js_2^j=s_2s_1=s_3j$.
Hence $s_2s_1^js_2^j=s_1=s_2^js_1^js_2$. 

Conjugating the left and the right part of the previous equation by $s_2$
we see that

\[ s_1^js_2^js_2=s_2s_2^js_1^j.\]

Multiplying both sides of the equation from the right by $s_1^j$ we have

\[ s_1^js_2^js_2s_1^j=s_2s_2^j.\]

Thus $s_1^j$ inverts $s_2s_2^j$. Conjugating by $j$ we
see that also $s_1$ inverts $s_2s_2^j$. Hence $s_1^js_1=js_1js_1=(js_1)^2$ centralizes $s_2^js_2=(js_2)^2$.
By assumption, $s_1$ and $s_2$ are equivalent relative to 
$j$. 
Since we also have $s_3s_2=js_1$, we see similarly  that
$s_2$ and $s_3$ are equivalent relative to $j$, proving the claim.
\end{proof}

\begin{lemma}\label{l:translation are conjugate}

 For any distinct involutions $s_1,s_2\in G$ with $s_1s_2\in\lan j,t\ran$ there is some some $g\in G$ with $s_1^g=j, s_2^g=t$.
\end{lemma}
\begin{proof} We may assume that $s_2\neq j$. Then $s_2$ is equivalent to $t$ by Lemma~\ref{l:translation_product}. By assumption (4) there is some $a\in A$ with $s_2^a=t$.
Note that $A^{t'}$ is the centralizer of $t$. Since the involutions in
$\lan j,t\ran\setminus \{t\}$ are equivalent to $j$ relative to $t$, they are conjugate under $A^{t'}$ by property (4) applied to $A^{t'}={\Cen}(t)$. Hence there is some $g\in A^{t'}$ with $(s_1^a)^g=j$ and $ag\in G$ is as
required.
\end{proof}

We will use the following easy facts:

\begin{lemma}\label{l:dihedral properties}
Suppose $v\in G$ is a minimal involution and $s$ is equivalent to $v$ relative to $j$. If $(js)^g\in \lan j, v\ran$ for some $g\in G$, then $g\in \lan j,v\ran$.
\end{lemma}
\begin{proof}
Since $jv$ centralizes $js$, we see that $(jv)^g$ centralizes $(js)^g\in \langle jv\rangle$. Hence $(jv)^g\in \langle jv\rangle$ showing that  $g$ normalizes $\langle jv\rangle$ and hence either centralizes or inverts $jv$. 
Now $j\in \langle j,v\rangle$ inverts $jv$ and ${\Cen}(jv)=\langle jv\rangle$, proving the claim.
\end{proof}

\begin{lemma}\label{l:linear}
For any involution $s\in G$ and $n>0$, there is at most one involution $s_1$ such that $(js_1)^n=js$. In particular,
$(js_1)^n=(js_2)^n$ implies $s_1=s_2$.
\end{lemma}
\begin{proof}
If $s\notin AtA$, this holds in $\lan j,s'\ran$ for $s'$ the minimal involution equivalent to~$s$. If $s\in AtA$, then $s$ is conjugate to $t$ by Lemma~\ref{l:invol} and the claim follows from  Property (3).
\end{proof}




\section{Background on HNN-extensions}\label{s:HNN}


The construction of $G_1$ in Proposition~\ref{p:main} will be given in Section~\ref{s:construction of G_1} as an HNN-extension of $G$. In preparation
for the proof we here collect some general facts about HNN-extensions.
We start with Britton's lemma, which we state for
an arbitrary group $G$ with isomorphic subgroups $D_t,D_v\leq G$ and a fixed isomorphism $f: D_t\longrightarrow D_v$.

Suppose $D_t,D_v\leq G$ and $G_1=\lan G, f\mid D_t^f=D_v\ran$ is an HNN extension of $G$. Then any element of $G_1$ has the form 
\[
g=g_0f^{\gd_1}g_1\cdots g_{m-1}f^{\gd_m}g_m,
\]
where $g_i\in G,\ i=0,\dots m,\ \gd_i=\pm 1,\ i=1,\dots m$.
We say that the expression for $g$ is \emph{reduced} if the equality
$\gd_i=-\gd_{i-1}=1$ implies  $g_i\notin D_t,$ and $\gd_i=-\gd_{i-1}=-1$ implies
$g_i\notin D_v$. Thus, an expression for $g$ is reduced if there are {\it no $f$-cancellations in $g$}, and in that case we call $m$ the \emph{length} of $g$.

Britton's lemma can be phrased in the following way:

\begin{thm}\label{rem hnn}\rm{(Britton's lemma)}\label{r:Britton alternative}
If $g\in G_1=\lan G, f\mid D_t^f=D_v\ran$ has a reduced expression
\[g=g_0f^{\gd_1}g_1\cdots g_{m-1}f^{\gd_m}g_m \]
then \emph{every} reduced expression for $g$ is of the form
\[g=g_0w_1f^{\gd_1}z_1^{-1}g_1w_2\cdots g_{m-1}w_mf^{\gd_m}z_{m}^{-1}g_m \]
where \[w_i\in D_t \mbox{ if } \gd_i=1 \mbox{ and }z_i=w_i^f\in D_v\]
and
\[w_i\in D_v \mbox{ if } \gd_i=-1 \mbox{ and }z_i=w_i^{f^{-1}}\in D_t.\]
In particular, the length of an element in $G_1$ is well-defined.
\end{thm}

\begin{remark}\label{r:cancellation}
Suppose  $G_1=\lan G, f\mid D_t^f=D_v\ran$ is an HNN extension of $G$ and $g,h\in G_1$ have
reduced expressions
\[g=g_0f^{\gd_1}g_1\cdots g_{m-1}f^{\gd_m}g_m \]
and
\[h=h_0f^{\eta_1}h_1\cdots h_{k-1}f^{\eta_k}h_k.\]

If $gh$ has length $m+k-n$, then it follows from Britton's lemma that $n$ is even and that we can rewrite $h$ as
\[h=g_m^{-1}f^{-\gd_m}g_{m-1}^{-1}\cdots g_{m-n/2}^{-1}f^{-\gd_{m-n/2}}h'_{n/2+1}\cdots h_{k-1}f^{\eta_k}h_k\]

where $h'_{n/2+1}=wh_{n/2+1}$ for some $w\in D_t\cup D_v$.
\end{remark}

As a special case we note the following for future reference:
\begin{remark}
Any involution $s\in G_1$ has a reduced expression
of the form
\[ s=g_1f^{\epsilon_1}g_2f^{\epsilon_2}\ldots g_m f^{\epsilon_m} s_1f^{-\epsilon_m}g_m^{-1}\ldots g_2^{-1}f^{-\epsilon_1}g_1^{-1} \]
for some involution $s_1\in G$.
\end{remark}
\begin{definition}
We call an element $g\in  G_1=\lan G, f\mid D_t^f=D_v\ran$ \emph{cyclically  reduced}
if any reduced expression for $g$ of the form
\[g=at^{\epsilon_1}g_2f^{\epsilon_2}\ldots f^{\epsilon_k}a^{-1}\]
implies $a=1$ and $\epsilon_1=\epsilon_k$.
\end{definition}

As a convenient abbreviation we will write
\[D_1= D_t \mbox{ and\ } D_{-1}=D_v\]
so that for $\epsilon\in\{-1, 1\}$ we have 
\[D_\epsilon f^\epsilon=f^\epsilon D_{-\epsilon}.\]

\begin{remark}\label{r:cycl red}
(i) Clearly, any $g\in G_1$ is conjugate to a cyclically reduced one.
In particular, any $g\in G$ is cyclically reduced.

(ii) If \[g=g_1f^{\epsilon_1}g_2f^{\epsilon_2}\ldots f^{\epsilon_k}g_{k+1} \] is cyclically reduced, then so is \[g^{g_{k+1}^{-1}}=g_{k+1}g_1f^{\epsilon_1}g_2f^{\epsilon_2}\ldots f^{\epsilon_k}.\] 
Otherwise we have $h=g_{k+1}g_1\in D_{\epsilon_1}$ and $\epsilon_1=-\epsilon_k$. Thus $g_{k+1}=hg_1^{-1}$ and we have a reduced expression for $g$ of the form
\[g=g_1f^{\epsilon_1}g_2f^{\epsilon_2}\ldots g_kh^{f^{-\epsilon_k}} f^{\epsilon_k}g_1^{-1}, \] contradicting our assumption that $g$ be cyclically reduced.

(iii) Note that if $g$ is cyclically reduced, then for any $n\in\mathbb{N}$ and any reduced expression for $g$,
we obtain a reduced expression for $g^n$ by concatenating $n$ copies of the reduced expression for $g$.
\end{remark}

\begin{lemma}\label{l:HNN malnormal}\label{l:centralizer in G}
Suppose $G_1=\lan G, f\mid D_t^f=D_v\ran$ and let $g\in G, h\in G_1$ be such that 
$g^h\in G$. Then we have $h\in G$ or there is some $b\in G$ with $g^b\in D_t\cup D_v$.
In particular, if the centralizer of $g$ is not contained in $G$, then  there is some $b\in G$ with $g^b\in D_t\cup D_v$.
\end{lemma}

\begin{proof}
Write $h\in G_1$ in reduced form,
\[h=h_1f^{\delta_1}h_2\ldots  f^{\delta_m}h_{m+1} .\]
If $h\notin G$ and $g^h\in G$, then $f$-cancellation
must occur at $f^{-\delta_1}h_1^{-1}gh_1f^{\delta_1}$
implying that $g^{h_1}\in D_t\cup D_v$ with $h_1\in G$.
\end{proof}


\section{Construction of $G_1$ and $A_1$ in Proposition~\ref{p:main}}\label{s:construction of G_1}


With the notation as in Theorem \ref{p:main} note that if there is an element 
$f\in A$ such that $t^f=v$, we can just take  $G_1=G$
and $A_1=A$ and there is nothing to prove in Proposition \ref{p:main}.  
Moreover, we may assume that $v$ is a minimal involution relative to $j$:

\begin{lemma}\label{l:reduction to minimal}
It suffices to consider the case that $v$ is a minimal
involution relative to $j$.
\end{lemma}
\begin{proof}
Suppose we have constructed $(G_1,A_1)$ as in Proposition~\ref{p:main} for a minimal involution $v$. Then $v$ is conjugate under $A_1$ to $t$.
Since all involutions equivalent to
$t$ are conjugate under $A\leq A_1$, all involutions equivalent to 
$v$ will be conjugate under $A_1$.
\end{proof}

So let $v$ be a minimal involution relative to $j$ and put
\[D_t=\langle j,t\rangle \mbox{ and } D_v=\langle j,v\rangle.\]
Note that $D_t$ and $D_v$ are infinite dihedral groups and hence isomorphic under an isomorphism $f$ fixing $j$ and taking $t$ to $v$.

We may now define the extension $G_1$  required in Proposition~\ref{p:main} as the HNN-extension
of $G$ by $f$ taking $D_t$ to $D_v$, i.e.

\[ G_1 =\lan G, f\mid j^f=j, t^f=v\ran.\]

Proposition~\ref{p:main} will be proved once we show that
$G_1$ and $A_1={\Cen}_{G_1}(j)$ satisfy the required properties in Proposition~\ref{p:main}. Clearly, $f\in A_1$ and $t^f=v$.

\section{Properties (1) -- (6) hold in $G_1$ }\label{s:1-6}

We have to verify that $G_1$ and $A_1={\Cen}_{G_1}(j)$ satisfy the conditions of Proposition~\ref{p:main}:

\begin{enumerate} 
\item[{(1)}] 
 all involutions in $G_1$ are conjugate;

\item[{(2)}] 
 any two distinct involutions of $G_1$ generate infinite dihedral groups;

\item[{(3)}] 
 for any $n>0$ there is a unique involution $s\in G_1$ with
$(js)^n=jt$.

\item[{(4)}]  
for any involution $s\in G_1$ equivalent to $t$  there is some $a\in A_1$ such that $s=t^a$.

\item[{(5)}]  
for any involution $s\in G_1$, we have ${\Cen}(js)=\Tr(\overline{\lan j,s\ran}^{G_1})$.

\item[{(6)}]  
 for any involution $s\notin A_1tA_1$ there is a
unique minimal involution $s'$ equivalent to $s$ relative to $j$.

\end{enumerate}

It is well-known that Properties (1) and (2) are preserved under HNN-extensions.
For the remaining properties, we collect some easy observations:

\begin{lemma}\label{l:invol in A1}
The involution $j$ is the only involution in $A_1$. 
\end{lemma}
\begin{proof}
Suppose $s\in A_1\setminus G$ is an involution. Write $s$ in reduced form as
\[s=g_1f^{\epsilon_1}g_2f\ldots g_m f^{\epsilon_m} s_1f^{-\epsilon_m}g_m^{-1}\ldots g_2^{-1}f^{-\epsilon_1}g_1^{-1}\]
for some involution $s_1\in G$.
Then
\[sjs=g_1f^{\epsilon_1}\ldots g_m f^{\epsilon_m} s_1f^{-\epsilon_m}g_m^{-1}\ldots f^{-\epsilon_1}g_1^{-1}jg_1f^{\epsilon_1}\ldots g_m f^{\epsilon_m} s_1f^{-\epsilon_m}g_m^{-1}\ldots f^{-\epsilon_1}g_1^{-1}=j \]

Clearly $f$-cancellation must occur at $f^{-\epsilon_1}g_1^{-1}jg_1f^{\epsilon_1}$. We
 conclude inductively that $j_1=j^{g_1f^{\epsilon_1}g_2f^{\epsilon_2}\ldots g_m f^{\epsilon_m}}\in D_{-\epsilon_m}$ and $s_1\in {\Cen}(j_1)$. By assumption on $G$ we have $s_1=j_1\in D_{-\epsilon_m}$, contradicting our assumption that the expression for $s\in G_1\setminus G$ was reduced.
\end{proof}

\begin{lemma}\label{l:nonconjugate}
 If $g\in D_t$ is a nontrivial translation, then $g^b\notin D_v$  for all $b\in G$. 
\end{lemma}
\begin{proof} 
By Lemma~\ref{l:translation are conjugate}  it suffices to consider $g=jt$. Suppose $(jt)^b\in D_v$.
By replacing $b$ by $bj$ if necessary, we may assume $(jt)^b=(jv)^n$ for some $n>0$.
By properties (3) and (4) there is some $a\in A$ be such that $(jt^a)^n=jt$, so $(jt)^{ab}=jv$ by Lemma~\ref{l:linear}.
By Lemma~\ref{l:translation are conjugate} there is some $g\in G$ such that
\[j^{b^{-1}a^{-1}g}=j \mbox{ and } v^{b^{-1}a^{-1}g}=t.\]
Then $h=g^{-1}ab\in {\Cen}(j)=A$ and $t^h=v$, contradicting the assumption on~$v$.
\end{proof}

\begin{lemma}\label{l:centralizer of translations}
For any translation $w\in\lan j,t\ran$ the centralizer of $w$ in $G_1$ coincides with the centralizer of $w$ in $G$.
In particular we have 
\[{\Cen}_G(w)=\Tr(\overline{\langle j,t\rangle}^G)=\Tr(\overline{\langle j,t\rangle}^{G_1} )={\Cen}_{G_1}(w).\]
\end{lemma}
\begin{proof}

By Lemma~\ref{l:translation are conjugate} we may assume $w=jt$.
Let $g\in {\Cen}_{G_1}(jt)\setminus G$ have a reduced expression
\[g=g_1f^{\epsilon_1}g_2f^{\epsilon_2}\ldots f^{\epsilon_k}g_{k+1}.\]
 
and suppose towards a contradiction that $k>0$.
Then 
\[jt=g^{-1}_{k+1}f^{-\epsilon_k}g^{-1}_k\ldots f^{-\epsilon_1}g_1^{-1}jt 
g_1f^{\epsilon_1}g_2f^{\epsilon_2}\ldots f^{\epsilon_k}g_{k+1}.\hspace{1.5cm} (*)\]

By Britton's lemma, cancellation must occur at
\[f^{-\epsilon_1}g_1^{-1}jt 
g_1f^{\epsilon_1}\]

implying that $g_1^{-1}jt g_1\in D_t$ by Lemma~\ref{l:nonconjugate}  and hence $\epsilon_1=1$. 

Thus we have \[f^{-\epsilon_1}g_1^{-1}jt 
g_1f^{\epsilon_1}\in D_v.\]

Again by Lemma~\ref{l:nonconjugate} we must have $k\geq 2$. 
Since all the $f$-occurrences in $(*)$ cancel, by Lemma~\ref{l:nonconjugate} and  Lemma~\ref{l:dihedral properties}  we must have $g_2\in D_v$ and $\epsilon_2=-1$,
contradicting the assumption that the expression for $g$ was
reduced.
\end{proof}

\begin{cor}\label{c:equivalence_class}
If $s\in G_1$ is equivalent to $t$ relative to $j$, then
$s\in G$.
\end{cor}
\begin{proof}
Let $s\in G_1$ be such that $(js)^m=(jt)^n$. Then $js$ centralizes $jt$, so $js\in G$ by Lemma~\ref{l:centralizer of translations} and hence $s\in G$.
\end{proof}
Clearly, Corollary~\ref{c:equivalence_class} implies that Properties (3) and (4) continue to hold in $G_1$:
\begin{cor}\label{c:property 5} 
(i) For any $n>0$ there is a unique involution $s\in G_1$ with
$(js)^n=jt$.

(ii) For any involution $s\in G_1$ equivalent to $t$  there is some $a\in A_1$ such that $s=t^a$.
 \qed
\end{cor}

We next verify Property (5) by a variant of the \emph{Euclidean algorithm}.
Recall that we write
\[D_1= D_t \mbox{ and\ } D_{-1}=D_v\]
so that for $\epsilon\in\{-1, 1\}$ we have 
\[D_\epsilon f^\epsilon=f^\epsilon D_{-\epsilon}.\]

We first need the following lemma:
\begin{lemma}\label{l:cycl reduced translation}
Let $r\in G, s\in G_1$ be involutions. Then there is a cyclically reduced
translation $r's'$ conjugate to $rs$ where $r'\in G, s'\in G_1$ are involutions.
\end{lemma}
\begin{proof}
Write $s\in G_1$ in reduced form
\[s= s_1f^{\epsilon_1}s_2f^{\epsilon_2}\ldots s_m f^{\epsilon_m} s'f^{-\epsilon_m}g_m^{-1}\ldots s_2^{-1}f^{-\epsilon_1}s_1^{-1}.\]

If $rs$ is not cyclically reduced, then $w_1=s_1^{-1}rs_1\in D_{\epsilon_1}$. Put 
$ z_1=w_1^{f^{\epsilon_1}}.$ Then
\[(rs)^{s_1f^{\epsilon_1}}=f^{-\epsilon_1}w_1f^{\epsilon_1}s_2f^{\epsilon_2}\ldots f^{-\epsilon_2}s_2^{-1}= z_1s_2f^{\epsilon_2}\ldots f^{-\epsilon_2} s_2^{-1}.\] 

By iterating this step, the length decreases and we end with a cyclically reduced conjugate of $rs$ in the required form.
\end{proof}

Now we can show that Property (5) continues to hold in $G_1$:
\begin{lemma}\label{l:euclid final}
For involutions $r\in G,s\in G_1$, we have ${\Cen}(rs)=\Tr(\overline{\lan r,s\ran}^{G_1})$.
\end{lemma}
\begin{proof}
By Lemma~\ref{l:cycl reduced translation} we may assume that $rs$ is cyclically reduced.
If $rs\in G$, the claim holds 
by Lemmas~ \ref{l:centralizer in G} and \ref{l:centralizer of translations} and assumption on $G$. 
Hence we may assume $rs\in G_1\setminus G$.

Write $rs$ as a reduced expression
\[rs=rg_1f^{\epsilon_1}g_2f^{\epsilon_2}\ldots g_m f^{\epsilon_m} s_1f^{-\epsilon_m}g_m^{-1}\ldots g_2^{-1}f^{-\epsilon_1}g_1^{-1} \]
for involutions $r,s_1\in G$. Note that by Remark~\ref{r:cycl red} we may also assume $g_1=1$ and so $r\notin D_{\epsilon_1}$.
Consider $h\in {\Cen}(rs)$.

If $h\in G$, then by Britton's lemma and the equality $rsh=hrs$ we have $h, h^r\in D_{\epsilon_1}$.
If $h$ is a translation,
then by Lemma~\ref{l:centralizer of translations} we have $h\in D_v$, so $\epsilon_1=-1$.  By Lemma~\ref{l:dihedral properties} again
$r\in D_v=D_{\epsilon_1}$, contradicting
our assumption on $rs$. 
Hence $h\in D_{\epsilon_1}$ is an involution centralizing $rs$.
Since
\[rhrsh=s=hsrhr,\]
we see that the translation $h^rh=rhrh$ is inverted by $s$ and hence centralized by $js\in G_1\setminus G$. By Lemma~\ref{l:centralizer of translations} we have $h^rh\in D_v$, so again ${\epsilon_1}=-1$.
Now $(hh^r)^r=h^rh$ and so $r$ inverts $hh^r$. Lemma~\ref{l:dihedral properties} implies $r\in D_v=D_{\epsilon_1}$, contradicting
our assumption on $r$. 

Hence $h\in G_1\setminus G$.  Since $rs$ is cyclically reduced, by considering $h^{-1}$ if necessary, we may assume that $h$ has no reduced expression starting with $f^{\epsilon_1}$. The lemma will follow from:

\medskip
\noindent
{\bf Claim:} There is an involution $t_1\in G_1\setminus G$ of length   $d=\gcd(2m,k)$ such that $h,rs\in \langle rt_1\rangle$.
\medskip

Write $h=h_1f^{\delta_1}h_2\ldots h_kf^{\delta_k}h_{k+1}$.
By assumption on $h$ and Remark~\ref{r:cancellation} the expression
\[rsh=rf^{\epsilon_1}g_2f^{\epsilon_2}\ldots g_2^{-1}f^{-\epsilon_1}h_1f^{\delta_1}h_2\ldots f^{\delta_k}h_{k+1}\]
is reduced and since $rsh=hrs$, by Britton's lemma so is
\[hrs=h_1f^{\delta_1}h_2\ldots  f^{\delta_k}h_{k+1}rf^{\epsilon_1}g_2f^{\epsilon_2}\ldots g_2^{-1}f^{-\epsilon_1}.\]

Again by Britton's lemma, we see that $\delta_1=\epsilon_1=-\delta_k$. In particular $k\geq 2$ and we may assume $h_1=r, h_{k+1}=1$. So  $h$ is cyclically reduced as well. 
Now consider 
\[rsh=rf^{\epsilon_1}g_2f^{\epsilon_2}\ldots g_2^{-1}f^{-\epsilon_1}rf^{\epsilon_1}h_2f^{\delta_2}\ldots h_kf^{-\epsilon_1}=\]
\[hrs=rf^{\epsilon_1}h_2f^{\delta_2}\ldots h_k f^{-\epsilon_1}rf^{\epsilon_1}g_2f^{\epsilon_2}\ldots g_2^{-1}f^{-\epsilon_1}.\]

Write $k= n\cdot 2m +l$ with $l<2m$.
We can rewrite $h$ according to Britton's lemma as
\[h=rf^{\epsilon_1}g_2\ldots f^{-\epsilon_1}rf^{\epsilon_1}\ldots  f^{-\epsilon_1} rf^{\epsilon_1}\ldots  f^{-\epsilon_1}rf^{\epsilon_1}g_2f^{\epsilon_2}\ldots f^{\epsilon_l} w.\]
where $w\in D_{-\epsilon_l}$.
If $l=0$, the claim follows. Otherwise $h'=(rs)^{-n}h \in {\Cen}(rs)$ has a reduced expression
\[h'=rf^{\epsilon_1}g_2f^{\epsilon_2}\ldots f^{\epsilon_l} w\]
where $\epsilon_l=-\epsilon_1,\epsilon_{l-1}=-\epsilon_2$ etc.

Since $r\notin D_{\epsilon_1}$, we see that  $h'\in{\Cen }(rs)$ is again cyclically reduced.
We can therefore iterate the procedure with $h', rs$ in place of $rs, h$, rewriting $rs$ in reduced form as a concatenation of copies of 
the reduced expression for $h'$ followed by an initial segment of $h'$ and $r$.

This procedure stops with a reduced expression of length $d=\gcd(m,k)\geq 2$ of the form
\[h''=rf^{\epsilon_1}g_2\ldots g_2^{-1}f^{-\epsilon_1} \in{\Cen}(rs).\]
So $h''=rt_1$ for some
involution $t_1\in G_1$  equivalent to $s$ and of length $d$ with and $h,rs\in\langle rt_1\rangle$.
\end{proof}

Note the following corollary of the proof:
\begin{cor}\label{c:translation centralizer}
If $s\in G_1\setminus G$ is such that $js$ is cyclically reduced, then there is an involution $t_1\in G_1\setminus G$ such that $jt_1$ is cyclically reduced and ${\Cen}(js)=\langle j t_1\rangle$. In particular, $t_1$ is the unique minimal involution equivalent to $s$.
\end{cor}
\begin{proof}
All the claims follow from the proof  of Lemma~\ref{l:euclid final}:
the proof shows that any element in the centralizer of $js$ has length at least $2$. Hence
an element of minimal length in the centralizer will yield the required minimal involution.
\end{proof}

Now we can also establish Property (6):

\begin{lemma}\label{l:6}
If $s\in G_1$ is an involution with $s\notin A_1tA_1$, then
there is a unique minimal involution $s'$ equivalent to $s$ relative to $j$.
\end{lemma}

\begin{proof}
By conjugating we may assume that $js$ is cyclically reduced. If $s\in G$, then the result holds
 by Lemmas~\ref{l:centralizer in G} and \ref{l:centralizer of translations}.
If $s\in G_1\setminus G$, then $s\notin A_1tA_1$  and this is contained in 
Corollary~\ref{c:translation centralizer}.
\end{proof}

\section{The proof of Theorems \ref{thm s2t in char 0} and \ref{t:existence}}\label{s:maintheorems}

In this section we show how Theorems \ref{thm s2t in char 0}
and \ref{t:existence} of the
introduction follows from Proposition \ref{p:main}.

In order to obtain a new group from $G$ and $A$ in Proposition~\ref{p:main}, we need that $G$ is not already sharply $2$-transitive. Thus to get started, we can use the following preparatory lemma:

\begin{lemma}\label{l:makenontrans}
If $G_0$ is a group containing involutions $j,t$ and $t'$ such that
the assumptions of Theorem~\ref{thm s2t in char 0} hold, then
$G=G_0\ast \zz$ also satisfies the assumptions. If $G_0$ acts sharply $2$-transitively on the set of its involutions, then $G$ is not
$2$-transitive on the set of its involutions.
\end{lemma}
\begin{proof}
It follows as before
that conditions (1), (2), (3) and (4)  are preserved.
It remains to verify (5) and (6). The proofs are exactly 
as the corresponding proofs in Section~\ref{s:1-6}.
Note that ${\Cen}_G(j)={\Cen}_{G_0}(j)$, and the set of involutions
of $G_0$ is strictly contained in the corresponding set of $G$. So the
last claim is immediate.
\end{proof}

\begin{proof}(of Theorem~\ref{thm s2t in char 0})
Let $G, A$ and $j, t$ and $t'$ be as in Theorem \ref{thm s2t in char 0}. If
$G$ is already sharply $2$-transitive, consider $G\ast \zz$ instead.
Set $G_0:=G,\ A_0:=A$. 
We now use Proposition~\ref{p:main} to construct an ascending sequence of groups $G_i$ such that
with $A_i={\Cen}_{G_i}(j)$ the groups $G_i, A_i$ satisfy properties (1) -- (6)
of Theorem~\ref{thm s2t in char 0} and such that 
for each $i\ge 0,$ and each involution $v_i\in G_i, v_i\neq j$ there
exists an element $f_i\in A_{i+1}$ such that $t^{f_i}=v_{i}$.

Suppose the groups $G_k$ and their subgroups $A_k$ were constructed for
$k\leq i$. We then construct $G_{i+1}$ as the union of groups obtained from  applying 
Proposition~\ref{p:main} to each minimal involution $v_i\in G_i, v_i\neq j$.
Then for any involution $v\in G_i$ there exists an element $f\in A_{i+1}={\Cen}_{G_{i+1}}(j)$ with $t^f=v$. 

Put
$\calg=\bigcup_{i=0}^{\infty}G_i$.
Then ${\Cen}_\calg(j)=\cala:=\bigcup_{i=0}^{\infty}A_i$
and $\calg$ acts sharply $2$-transitive on the set of its involutions.
It is left to show that $\calg$ does not contain a non-trivial abelian
normal subgroup. Otherwise by a result of B. H. Neumann \cite{Neumann} such a subgroup would be of the form $\Tr(\calg)$, which is nonabelian by property (4)
whenever $\calg$ contains distinct involutions $j,s,t$ with $t$ not normalizing $\langle j,s\rangle$ and these exist in $\calg$ by construction.
\end{proof}

We finish the paper with the proof of Theorem~\ref{t:existence}:
\begin{proof}(of Theorem~\ref{t:existence})
The sharply $2$-transitive group $AGL(1,\qq)\cong  \qq^+\rtimes \qq^*$ and $A=\qq^*$
satisfies the assumptions of Theorem~\ref{thm s2t in char 0}. By Lemma~\ref{l:makenontrans} we may therefore apply Theorem~\ref{thm s2t in char 0} to
\[G=AGL(1,\qq)\ast \zz\]
and obtain a sharply $2$-transitive group without abelian normal subgroup.
\end{proof}


\end{document}